\documentclass[11pt]{article}%
\usepackage[intlimits]{amsmath}
\usepackage{amssymb}
\usepackage[T1]{fontenc}
\usepackage[sc]{mathpazo}
\usepackage{color}
\usepackage[colorlinks]{hyperref}
\usepackage{amsfonts}
\usepackage{graphicx}%
\definecolor {refcol}{RGB}{40,0,255}
\hypersetup{colorlinks=true,allcolors=refcol}
\makeatletter
\newfont{\footsc}{cmcsc10 at 8truept}
\newfont{\footbf}{cmbx10 at 8truept}
\newfont{\footrm}{cmr10 at 10truept}
\newtheorem{theorem}{Theorem}

\newtheorem{lemma}[theorem]{Lemma}

\newtheorem{question}[theorem]{Question}

\newenvironment{proof}[1][Proof]{\noindent{\textbf {#1}  }}  {\hfill$\Box$\bigskip}
\begin{document}

\title{\textbf{The spectral radius of graphs with no }$K_{2,t}$ \textbf{minor }}
\author{V. Nikiforov\thanks{Department of Mathematical Sciences, University of
Memphis, Memphis TN 38152, USA; \textit{email: vnikifrv@memphis.edu}}}
\date{}
\maketitle

\begin{abstract}
Let $t\geq3$ and $G$ be a graph of order $n,$ with no $K_{2,t}$ minor. If
$n>400t^{6}$, then the spectral radius $\mu\left(  G\right)  $ satisfies%
\[
\mu\left(  G\right)  \leq\frac{t-1}{2}+\sqrt{n+\frac{t^{2}-2t-3}{4}},
\]
with equality if and only if $n\equiv1$ $(\operatorname{mod}$ $t)$ and
$G=K_{1}\vee\left\lfloor n/t\right\rfloor K_{t}$.

For $t=3$ the maximum $\mu\left(  G\right)  $ is found exactly for any
$n>40000.$\bigskip



\end{abstract}

\section{Introduction and main results}

A graph $H$ is called a \emph{minor} of a graph $G$ if $H$ can be obtained by
contracting edges of a subgraph of $G$. Write $H\nprec G$ if  $H$ is not a
minor of $G$. The \emph{spectral radius} $\mu\left(  G\right)  $ of a graph
$G$ is the largest eigenvalue of its adjacency matrix. In this note we study
the following question:\medskip

\begin{question}
\label{q1}How large can $\mu\left(  G\right)  $ be if $G$ a graph of order $n$
and $K_{2,t}\nprec G$? \medskip
\end{question}

Particular cases of this question have been studied before: for example, Yu,
Shu and Hong \cite{YSH12} showed that if $G$ is a graph of order $n$ and
$K_{2,3}\nprec G$, then
\begin{equation}
\mu\left(  G\right)  =3/2+\sqrt{n-7/4}.\label{x}%
\end{equation}
Unfortunately, bound (\ref{x}) it is not attained for any $G$, although it is
tight up to an additive term approaching $1/2$. Likewise, Benediktovich
\cite{Ben15} studied $2$-connected graphs with no $K_{2,4}$ minors and gave a
few bounds similar to (\ref{x}), but gave no summary result.

To outline the case $t=2$, let $n$ be odd and  $F_{2}\left(  n\right)  $ be
the \emph{friendship graph, }that is, a set of $\left\lfloor n/2\right\rfloor
$ triangles sharing a single common vertex. If $n$ is even, let $F_{2}\left(
n\right)  $ be obtained by hanging an extra edge to the common vertex of
$F_{2}\left(  n-1\right)  $.

In \cite{Nik07} and \cite{ZhWa12}, it was shown that if $G$ is a graph of
order $n$, with no $K_{2,2}$, then $\mu\left(  G\right)  <\mu\left(
F_{2}\left(  n\right)  \right)  $, unless $G=F_{2}\left(  n\right)  .$

Incidentally,  $K_{2,2}\nprec F_{2}\left(  n\right)  $; thus, Question
\ref{q1} is settled for $t=2.$ We shall show that the situation is similar for
any $t\geq3$ and $n$ large.

First, we extend the family $\left\{  F_{2}\left(  n\right)  \right\}  $ for
$t>2$. Given graphs $F$ and $H$, write $F\vee H$ for their \emph{join} and
$F+H$ for their \emph{disjoint union}. Suppose that $t\geq3$ and $n\geq t+1$;
set $p=\left\lfloor \left(  n-1\right)  /t\right\rfloor $ and let $n-1=pt+s$.

Now, let $F_{t}\left(  n\right)  :=K_{1}\vee\left(  pK_{t}+K_{s}\right)  $; in
particular, if $s=0,$ let $F_{t}\left(  n\right)  :=K_{1}\vee pK_{t}$.
Clearly, the graph $F_{t}\left(  n\right)  $ is of order $n$ and
$K_{2,t}\nprec F_{t}\left(  n\right)  $. 

It is not hard to find that $\mu\left(  F_{t}\left(  n\right)  \right)  $ is
the largest root of the cubic equation%
\[
\left(  x-s+1\right)  (x^{2}-\left(  t-1\right)  x-n+1)+s\left(  t-s\right)
=0,
\]
and satisfies the inequality
\[
\mu\left(  F_{t}\left(  n\right)  \right)  \leq\frac{t-1}{2}+\sqrt
{n+\frac{t^{2}-2t-3}{4}},
\]
with equality if and only if $n\equiv1$ $(\operatorname{mod}$ $t)$, i.e., if
$s=0$.

Our first result answers Question \ref{q1} for $t=3$ and $n$ large:

\begin{theorem}
\label{t1}If $G$ is a graph of order $n>40000$ and $K_{2,3}\nprec G$, then
$\mu\left(  G\right)  <\mu\left(  F_{3}\left(  n\right)  \right)  $, unless
$G=F_{3}\left(  n\right)  .$
\end{theorem}

A similar theorem may hold also for $t>3$, but our general result is somewhat weaker:

\begin{theorem}
\label{t2}Let $t\geq4$ and $n\geq400t^{6}$. If $G$ is a graph of order $n$ and
$K_{2,t}\nprec G$, then
\begin{equation}
\mu\left(  G\right)  \leq\frac{t-1}{2}+\sqrt{n+\frac{t^{2}-2t-3}{4}%
}.\label{xx}%
\end{equation}
Equality holds if and only if $n\equiv1$ $(\operatorname{mod}$ $t)$ and
$G=F_{t}\left(  n\right)  .$
\end{theorem}

Before proving these theorems, let us note that if $t\geq4$ and $n\geq
400t^{6}$, then
\[
\mu\left(  F_{t}\left(  n\right)  \right)  >\frac{t-1}{2}+\sqrt{n+\frac
{t^{2}-2t-3}{4}}-\frac{t\left(  t+1\right)  }{8n},
\]
so bound (\ref{xx}) is quite tight.

The proofs of Theorems \ref{t1} and \ref{t2} are based on a structural lemma
inspired by \cite{TaTo16}. Write $\mathcal{M}_{t}\left(  n\right)  $ for the
set of graphs of order $n$, with no $K_{2,t}$ minors, and with maximum
spectral radius.

\begin{lemma}
\label{le1}Let $t\geq3$, $n>16\left(  t-1\right)  ^{4}\left(  5t-3\right)
^{2}$, and $G\in\mathcal{M}_{t}\left(  n\right)  $. If $\mathbf{x}$ is an
eigenvector to $\mu\left(  G\right)  $, then the maximum entry of $\mathbf{x}$
corresponds to a vertex of degree $n-1$.
\end{lemma}

\begin{proof}
Let $t,n$ and $G$ be as required. Hereafter, let $V:=$ $\left\{  v_{1}%
,\ldots,v_{n}\right\}  $ be the vertex set of $G$; let $\Gamma_{G}\left(
v\right)  $ be the set of the neighbors of $v\in V$, and set $d_{G}\left(
v\right)  :=\left\vert \Gamma_{G}\left(  v\right)  \right\vert $; the
subscript $G$ is omitted if $G$ is understood. Also, $G-v$ stands for the
graph obtained by omitting the vertex $v$.

Clearly $G$ is connected, as otherwise $G$ there is a graph $H$ with no
$K_{2,t}$ minor such that  $\mu\left(  H\right)  >\mu\left(  G\right)  $,
contradicting $G\in\mathcal{M}_{t}\left(  n\right)  $.

Set for short, $\mu:=\mu\left(  G\right)  $, and let $\mathbf{x}:=\left(
x_{1},\ldots,x_{n}\right)  $ be a unit eigenvector to $\mu$ such that
$x_{1}\geq\cdots\geq x_{n}$. We have to show that $d\left(  v_{1}\right)
=n-1$.

Let $A$ be the adjacency matrix of $G$ and set $B=\left[  b_{i,j}\right]
:=A^{2}.$ Note that $b_{i,j}$ is equal to the number of $2$-walks starting at
$v_{i}$ and ending at $v_{j}$; hence, if $i\neq j$, then $b_{i,j}\leq t-1$, as
$K_{2,t}\nsubseteq G$. Since $B\mathbf{x}=\mu^{2}\mathbf{x}$, for any vertex
$u$, we see that
\[
\mu^{2}x_{u}=d\left(  u\right)  x_{u}+\sum\limits_{i\in V\backslash\left\{
u\right\}  }b_{u,i}x_{i}\leq d\left(  u\right)  x_{u}+\left(  t-1\right)
\sum\limits_{i\in V\backslash\left\{  u\right\}  }x_{i}\leq d\left(  u\right)
x_{u}+\left(  t-1\right)  \left(  \sqrt{n}-x_{u}\right)  .
\]
The last inequality follows from $\left(  x_{1}+\cdots+x_{n}\right)  ^{2}\leq
n\left(  x_{1}^{2}+\cdots+x_{n}^{2}\right)  =n$. We find that
\begin{equation}
d\left(  u\right)  \geq\mu^{2}+t-1-\frac{\left(  t-1\right)  \sqrt{n}}{x_{u}}.
\label{bo}%
\end{equation}
On the other hand, if $u\neq v_{1}$, then \ $d\left(  u\right)  +d\left(
v_{1}\right)  \leq n+t-1$, as $K_{2,t}\nsubseteq G$. Using (\ref{bo}), we get%
\[
n+t-1\geq2\mu^{2}+2\left(  t-1\right)  -\frac{\left(  t-1\right)  \sqrt{n}%
}{x_{u}}-\frac{\left(  t-1\right)  \sqrt{n}}{x_{1}}\geq2\mu^{2}+2\left(
t-1\right)  -\frac{2\left(  t-1\right)  \sqrt{n}}{x_{u}}.
\]
In view of $\mu^{2}>n-1$, we obtain
\begin{equation}
x_{u}\leq\frac{2\left(  t-1\right)  \sqrt{n}}{2\mu^{2}-n+t-1}<\frac{2\left(
t-1\right)  \sqrt{n}}{n}=\frac{2\left(  t-1\right)  }{\sqrt{n}}. \label{uu}%
\end{equation}

Assume for a contradiction that $d\left(  v_{1}\right)  \leq n-2$; let $H$ be
the graph induced in $G$ by the set $V\backslash\left(  \Gamma\left(
v_{1}\right)  \cup\left\{  v_{1}\right\}  \right)  $ and suppose that $v$ is a
vertex with minimum degree in $H$. Since $K_{2,t}\nprec G$, Theorem 1.1 of
\cite{CRS11}\footnote{We use Theorem 1.1 of \cite{CRS11} solely to lower the
bound on $n$; otherwise just as good is an older result of Mader \cite{Mad72}
implying that $\left\vert E(G\right\vert \geq(4r+8)n$ forces a $K_{2,r}$
minor.} implies that $d_{H}\left(  v\right)  \leq t$, and since $v$ and
$v_{1}$ have at most $t-1$ common neighbors, we see that $d_{G}\left(
v\right)  \leq2t-1$.

Next, remove all edges incident to $v$ and join $v$ to $v_{1}$. Write
$G^{\prime}$ for the resulting graph, which is of order $n$ and $K_{2,t}\nprec
G^{\prime}$. As $G\in\mathcal{M}_{t}\left(  n\right)  $, we see that
\[
0\leq\mu-\mu\left(  G^{\prime}\right)  \leq2x_{u}\sum\limits_{i\in
\Gamma\left(  v\right)  }x_{i}-2x_{1}x_{u}.
\]
Thus, bound (\ref{uu}) implies an upper bound on $x_{1}$
\begin{equation}
x_{1}\leq\sum\limits_{i\in\Gamma\left(  v\right)  }x_{i}\leq d_{G}\left(
v\right)  \frac{2\left(  t-1\right)  }{\sqrt{n}}\leq\frac{2\left(  t-1\right)
\left(  2t-1\right)  }{\sqrt{n}}.\label{u1}%
\end{equation}
Finally, we apply (\ref{uu}) and (\ref{u1}) to show that $\mu$ is bounded in
$n$%
\begin{align*}
\mu &  =2\sum\limits_{\left\{  i,j\right\}  \in E\left(  G\right)  }x_{i}%
x_{j}\leq2x_{1}\sum\limits_{i\in\Gamma\left(  v_{1}\right)  }x_{i}%
+2\sum\limits_{\left\{  i,j\right\}  \in E\left(  G-v_{1}\right)  }x_{i}%
x_{j}\\
&  \leq\frac{8\left(  t-1\right)  ^{2}\left(  2t-1\right)  d\left(
v_{1}\right)  }{n}+\frac{8\left(  t-1\right)  ^{2}\left(  \left\vert E\left(
G\right)  \right\vert -d\left(  v_{1}\right)  \right)  }{n}\\
&  =\frac{16\left(  t-1\right)  ^{3}d\left(  v_{1}\right)  }{n}+\frac
{8\left\vert E\left(  G\right)  \right\vert \left(  t-1\right)  ^{2}}{n}.
\end{align*}
Since $d\left(  v_{1}\right)  <n,$ and Theorem 1.1 of \cite{CRS11} gives
$2\left\vert E\left(  G\right)  \right\vert \leq\left(  t+1\right)  \left(
n-1\right)  $, we find that
\[
n-1<\mu^{2}<16\left(  t-1\right)  ^{4}\left(  5t-3\right)  ^{2},
\]
contradicting the premises. Hence, $d\left(  v_{1}\right)  =n-1$.
\end{proof}

\begin{proof}
[\textbf{Proof of Theorem \ref{t2}}]Let $G\in\mathcal{M}_{t}\left(  n\right)
$, $\mu:=\mu\left(  G\right)  $, and $\mathbf{x}:=\left(  x_{1},\ldots
,x_{n}\right)  $ be a unit eigenvector to $\mu$ such that $x_{1}\geq\cdots\geq
x_{n}$. Lemma \ref{le1} implies that $d\left(  v_{1}\right)  =n-1$. Clearly
$\mu x_{1}\leq\left(  n-1\right)  x_{2}$ and since $d\left(  v_{2}\right)
\leq t$, we see that $\mu x_{2}\leq x_{1}+\left(  t-1\right)  x_{2}.$
Therefore,%
\[
\mu\left(  \mu-t+1\right)  \leq n-1,
\]
implying (\ref{xx}). If equality holds in (\ref{xx}), then $x_{2}=x_{3}%
=\cdots=x_{n}$ and $\mu x_{2}=x_{1}+d\left(  u\right)  x_{2}$ for
$u=2,\ldots,n$. Hence, $G-v_{1}$ is $\left(  t-1\right)  $-regular. To
complete the proof, we show that $G-v_{1}$ is a union of disjoint $K_{t}$s.

Assume for a contradiction that $G-v_{1}$ has a component $H$ that is
non-isomorphic to $K_{t}$, and let $h$ be the order of $H$. Clearly $h\geq
t+2$, for if $h=t+1$, any two nonadjacent vertices in $H$ have $t-1$ common
neighbors, which together with $v_{1}$ form a $K_{2,t}$.

Further, since $K_{2,t}\nprec G$, we see that $K_{1,t}\nprec H$. As shown in
\cite{DJS01}\footnote{See also Section 1.2 of \cite{CRS11} where the result is
stated more fittingly for our use.}, these conditions on $H$ imply that
$\left\vert E\left(  H\right)  \right\vert \leq h+t\left(  t-3\right)  /2$,
contradicting the identity $\left\vert E\left(  H\right)  \right\vert =\left(
t-1\right)  h/2$. Hence, $G-v_{1}$ is a union of disjoint $K_{t}$s, completing
the proof of Theorem \ref{t2}.
\end{proof}

\begin{proof}
[\textbf{Proof of Theorem \ref{t1}}]Let $G\in\mathcal{M}_{t}\left(  n\right)
$, $\mu:=\mu\left(  G\right)  $, and $\mathbf{x}:=\left(  x_{1},\ldots
,x_{n}\right)  $ be a unit eigenvector to $\mu$ such that $x_{1}\geq\cdots\geq
x_{n}$. Lemma \ref{le1} implies that $d\left(  v_{1}\right)  =n-1$. Since
$G-v_{1}$ has no vertex of degree more than $2$, its components are paths,
triangles, or isolated vertices, as otherwise $G$ contains a $K_{2,3}$ minor.

Since $G-v_{1}$ is edge maximal, it may have at most one component that is not
a triangle, say the component $H$. If $H$ is an isolated vertex or an edge, we
are done, so suppose that $H$ is a path of order $h$, and let $v_{k+1}%
,\ldots,v_{k+h}$ be the vertices along the path. Clearly $h\geq4$.

Suppose first that $h$ is odd, say $h=2s+1$ and $s\geq2.$ By symmetry,
$x_{k+i}=x_{k+h-i+1}$ for any $i\in\left[  s\right]  $. Remove the edges
$\left\{  v_{k+s-1},v_{k+s}\right\}  $, $\left\{  v_{k+s+2},v_{k+s+3}\right\}
$; add the edges $\left\{  v_{k+s},v_{k+s+2}\right\}  $, $\left\{
v_{k+s-1},v_{k+s+3}\right\}  $; and write $G^{\prime}$ for the resulting
graph. Clearly $K_{2,3}\nprec G^{\prime}$ has no $K_{2,3}$ minor, as $H$ is
replaced by a shorter path and a disjoint triangle. On the other hand,
\begin{align*}
\sum\limits_{\left\{  i,j\right\}  \in E\left(  G^{\prime}\right)  }x_{i}x_{j}
&  =\sum\limits_{\left\{  i,j\right\}  \in E\left(  G\right)  }x_{i}%
x_{j}-x_{k+s-1}x_{k+s}-x_{k+s+2}x_{k+s+3}+x_{k+s}x_{k+s+2}+x_{k+s-1}%
x_{k+s+3}\\
&  =\sum\limits_{\left\{  i,j\right\}  \in E\left(  G\right)  }x_{i}%
x_{j}+\left(  x_{k+s-1}-x_{k+s}\right)  ^{2}.
\end{align*}
Since $G\in\mathcal{M}_{t}\left(  n\right)  $, we get $\mu\left(  G^{\prime
}\right)  =\mu$; hence $\mathbf{x}$ is an eigenvector to $\mu\left(
G^{\prime}\right)  $. But $v_{k+s}$, $v_{k+s+1}$, and $v_{k+s+2}$ are
symmetric in $G^{\prime}$, implying that $x_{k+s}=x_{k+s+1}=x_{k+s+2}$. Now,
using the eigenequations of $G,$ we find that $x_{k+1}=\cdots=x_{k+h}$, which
is a contradiction, in view of
\[
\mu x_{k+1}=x_{k+2}+x_{1}\text{ \ \ \ and \ \ \ }\mu x_{k+2}=x_{k+3}%
+x_{k+1}+x_{1}\text{. }%
\]

Next, suppose that $h$ is even, say $h=2s$, and let $s\geq3$. By symmetry,
$x_{k+i}=x_{k+h-i+1}$ for any $i\in\left[  s\right]  $. Remove the edges
$\left\{  v_{k+s-1},v_{k+s}\right\}  $, $\left\{  v_{k+s+2},v_{k+s+3}\right\}
$; add the edges $\left\{  v_{k+s},v_{k+s+2}\right\}  $, $\left\{
v_{k+s-1},v_{k+s+3}\right\}  $; and write $G^{\prime}$ for the resulting
graph. Clearly $K_{2,3}\nprec G^{\prime}$, as $H$ is replaced by a shorter
path and a disjoint triangle. On the other hand,
\begin{align*}
\sum\limits_{\left\{  i,j\right\}  \in E\left(  G^{\prime}\right)  }x_{i}x_{j}
&  =\sum\limits_{\left\{  i,j\right\}  \in E\left(  G\right)  }x_{i}%
x_{j}-x_{k+s-1}x_{k+s}-x_{k+s+2}x_{k+s+3}+x_{k+s}x_{k+s+2}+x_{k+s-1}%
x_{k+s+3}\\
&  =\sum\limits_{\left\{  i,j\right\}  \in E\left(  G\right)  }x_{i}x_{j}.
\end{align*}
Since $G\in\mathcal{M}_{t}\left(  n\right)  $, we get $\mu\left(  G^{\prime
}\right)  =\mu$; hence $\mathbf{x}$ is an eigenvector to $\mu\left(
G^{\prime}\right)  $. But $v_{k+s}$, $v_{k+s+1}$, and $v_{k+s+2}$ are
symmetric in $G^{\prime}$, implying that $x_{k+s}=x_{k+s+1}=x_{k+s+2}.$ This
fact leads to a contradiction precisely as above.

It remains the case $h=4.$ By symmetry, $x_{k+1}=x_{k+4}$ and $x_{k+2}%
=x_{k+3}.$ Remove the edge $\left\{  v_{k+1},v_{k+2}\right\}  $, add the edge
$\left\{  v_{k+2},v_{k+4}\right\}  $, and write $G^{\prime}$ for the resulting
graph. Clearly $K_{2,t}\nprec G^{\prime}$ and
\[
\sum\limits_{\left\{  i,j\right\}  \in E\left(  G^{\prime}\right)  }x_{i}%
x_{j}=\sum\limits_{\left\{  i,j\right\}  \in E\left(  G\right)  }x_{i}%
x_{j}-x_{k+1}x_{k+2}+x_{k+2}x_{k+4}=\sum\limits_{\left\{  i,j\right\}  \in
E\left(  G\right)  }x_{i}x_{j}.
\]
Since $G\in\mathcal{M}_{t}\left(  n\right)  $, we get $\mu\left(  G^{\prime
}\right)  =\mu$; hence $\mathbf{x}$ is an eigenvector to $\mu\left(
G^{\prime}\right)  $, implying the contradicting eigenequations
\[
\mu\left(  G^{\prime}\right)  x_{k+1}=x_{1}\text{ \ \ and \ \ }\mu\left(
G^{\prime}\right)  x_{k+4}=x_{k+2}+x_{k+3}+x_{1}.
\]
The proof of Theorem \ref{t1} is completed.
\end{proof}

\end{document}